\documentclass[reqno,a4paper]{amsart}
\usepackage{amsfonts}
\usepackage{amssymb, amsmath,amsthm,enumerate}
\usepackage[mathscr]{eucal}

\theoremstyle{plain}
\newtheorem{thm}{Theorem}[section]
\newtheorem{fact}[thm]{Fact}
\newtheorem{cor}[thm]{Corollary}
\newtheorem{lem}[thm]{Lemma}
\theoremstyle{definition}

\newtheorem{exa}[thm]{Example}
\newtheorem{op}[thm]{Open Problem}
\newtheorem*{FGP}{The First Gap Lemma}
\newtheorem*{SGP}{The Second Gap Lemma}
\newtheorem*{TGP}{The Third Gap Lemma}

\begin{document}
\title{The center of distances of some multigeometric series}
\author[M. Banakiewicz, A. Bartoszewicz \and F. Prus-Wi\'{s}niowski]{Micha\l\ Banakiewicz, Artur Bartoszewicz \and Franciszek Prus-Wi\'{s}niowski}

\newcommand{\acr}{\newline\indent}

\address{\llap{*\,}Studium of Mathematics\acr
 West Pomeranian University of Technology\acr
Al. Piast\'{o}w 48\acr
 PL-70-311 Szczecin\acr
  Poland}
\email{mbanakiewicz@zut.edu.pl}

\address{\llap{*\,}Institute of Mathematics\acr
\L\'{o}d\'{z} University of Technology\acr
ul. w\'{o}lcza\'{n}ska 215\acr
PL-93-005 \L\'{o}d\'{z}\acr
Poland}
\email{artur.bartoszewicz@p.lodz.pl}

\address{\llap{*\,}Instytut Matematyki\acr
Uniwersytet Szczeci\'{n}ski\acr
ul. Wielkopolska 15\acr
PL-70-453 Szczecin\acr
Poland}
\email{franciszek.prus-wisniowski@usz.edu.pl}


\begin{abstract}
Basic properties of the center of distances of a set are investigated. Computation of the center for achievement sets is particularly aimed at. A new sufficient condition for the center of distances of the set of subsums of a fast convergent series to consist of only 0 and the terms of the series is found. A complete description of the center of distances for some particular type of multigeometric series is provided. In particular, the center of distances $C(E)$ is the union of all centers of distances of $C(F_n)$ where $F_n$ is the set of all $n$-initial subsums of the discussed multigeometric series satisfying some special requirements. Several open questions are raised.
\end{abstract}

\maketitle
The notion of the center of distances of a set in a metric space was introduced recently by Bielas, Plewik and Walczy\'{n}ska \cite{BPW} and successfully applied for an impossibility proof. Using the properties of the center of distances, they proved that a special set $Z$ closely related to the Guthrie-Nymann Cantorval is not the set of subsums of any series. Their result is more interesting when phrased in the language of purely atomic probability measures (cf. \cite{Fer84}). Such a direction of research was the main topic of the paper \cite{BGM} with many interesting results, some still awaiting completion. Shortly speaking, Bielas, Plewik and Walczy\'{n}ska invented a new technique for proving that a given set is not the range of any purely atomic probabilistic measure. They have realized that the computation of the center of distances - even for relatively simple sets - is not an easy task because it requires skillful use of fractions. Since a complete characterization of which sets are ranges of purely atomic finite measures and which are not is of great interest to a number of mathematicians, we have decided to investigate the new concept of the center of distances.

It is worth mentioning  that the center of distances, although the first, is not the only tool for proving that a certain set is not the range of a purely atomic finite measure. Answering a question raised by W. Kubi\'{s} in 2015, the authors of \cite{BGFPS} proved that all members of a special family of geometric Cantorvals are not ranges of such measures by a different method.

Our main interest in not on the general case, but on computing centers of distances for sets of subsums of convergent series of positive terms. Given a nonincreasing squence $(a_n)$ of positive numbers converging to 0, the set
$$
E\ =\ E(a_n)\ :=\ \left\{x\in\mathbb R:\ \exists A\subset\mathbb N\ \ x=\,\sum_{n\in A}a_n\:\right\}
$$
is called \textsl{the achievement set} of the sequence $(a_n)$ or \textsl{the set of subsums} of the series $\sum a_n$. Throughout this note, by a series $\sum a_n$ we will always understand a convergent seris $\sum a_n$ with positive and nonincreasing terms, although, in general, the convergence of the series or even the convergence of $(a_n)$ to 0 are not necessary to define and investigate the set $E(a_n)$ (\cite{Na}, \cite{N}, \cite{P}, \cite[Thm. 21.3]{BFP}). The $k$-th remainder of the series will be denoted by $r_k:=\sum_{n>k}a_n$. In particular, $r_0=s$, where $s$ denotes the sum of the series. We will use the symbol $s$ for denoting the sum of the series and for no other purpose. The set of subsums of the $k$-th remainder series $\sum_{n=k+1}^\infty a_n$ will be denoted by
$$
E_k\ =\ E_k(a_n)\ :=\ \left\{x\in\mathbb R:\ \exists\ A\subset\{k+1,\,k+2,\,\ldots\,\}\ \ x\,=\,\sum_{n\in A}a_n\;\right\}.
$$
The set of $k$-initial subsums of the series will be denoted by
$$
F_k\ =\ F_k(a_n)\ :=\ \left\{x\in\mathbb R:\ \exists\ A\subset\{1,\,\ldots\,\,k\,\}\ \ x\,=\,\sum_{n\in A}a_n\;\right\}.
$$
The set $F=F(a_n):= \bigcup_{k\in\mathbb N} F_k$ is then the set of all finite subsums or, more precisely, of all sums of finite subseries.
It is obvious that
$$
E\ =\ F_k\;+\;E_k\ =\ \bigcup_{f\in F_k}(f+E_k)
$$
for any $k\in\mathbb N$. Clearly, $E_{k+1}\subset E_k$ and $F_k\subset F_{k+1}$ for all $k$.

The complete characterization of all possible topological types of sets of subsums is known. A set $P\subset\mathbb R$ is said to be a \textsl{multi-interval set} if it is the union of a finite family of closed and bounded intervals. A set $P\subset\mathbb R$ is said to be a \textsl{Cantor set} if it is nonempty, bounded, perfect and nowhere dense, that is, if it is homeomorphic to the classic Cantor ternary set. A set $P\subset\mathbb R$ is said to be a \textsl{Cantorval} if it is homeomorphic to the set
$$
GN\ :=\ C\ \cup\ \bigcup_n G_{2n-1}
$$
where $C$ denotes the classic Cantor ternary set and $G_{2n-1}$ is the union of all $4^{n-1}$ open intervals removed at the $(2n-1)$-st step of the standard geometric construction of $C$. It is known that a Cantorval is exactly a nonempty compact set in $\mathbb R$ such that  it is the closure of its interior and both endpoints of every nontrivial component are accumulation points of its trivial components. Other topological characterizations of Cantorvals can be found in \cite{MO} and \cite{BFP}. The well-known Guthrie-Nymann classification theorem (\cite[Thm. 1]{GN88},\cite{P}) asserts that the set $E$ of subsums of a series $\sum a_n$ is either a multi-interval set or a Cantor set, or a Cantorval. An analytic characterization of the above cases remains an open and a very challenging question for more than 30 years now. Only partial results in that direction are known.

A series $\sum a_n$ is said to be \textsl{slowly convergent} if $a_n\le r_n$ for all $n$. The set $E$ is a single interval if and only if the series $\sum a_n$ is slowly convergent. Even more, the set $E$ is a multi-interval set if and only if $a_n\le r_n$ for all sufficiently large $n$ (\cite{Kak14}, \cite{Kak14a}, \cite{P}).

A series $\sum a_n$ is said to be fast convergent if $a_n>r_n$ for all $n$. Fast convergence of $\sum a_n$ is a sufficient condition for $E$ to be a Cantor set (\cite{Kak14}, \cite{Kak14a}, \cite{P}). Some other sufficient conditions for $E$ to be a Cantor set can be found in \cite{BBFS}, \cite{BP} and \cite{BFS}.

Only a few sufficient conditions for $E$ to be a Cantorval are known (\cite{BBFS}, \cite{BFS} and \cite{Na}) and they all are formulated for multigeometric series. A series $\sum a_n$ will be called \textsl{multigeometric} if there are positive integers $m$ and $l_1\ge l_2\ge\ldots\ge l_m$, and a number $q\in(0,\,1)$ such that
$$
a_{km+i}\ =\ l_iq^k \hspace{.3in}\text{for all $k\in\mathbb N_0$ and all $i\in\{1,\ldots,m\,\}$.}
$$
It will be denoted by $\sum(l_1,\,\ldots,\,l_m;q)$. The multigeometric series play a great role in the  rapidly developing study of algebraic and topological properties of achievement sets because their sets of subsums are, in most cases, relatively easy to build and analyze from the computational point of view (\cite{N}, \cite{BFS}, \cite{BBFS}).

Given a nonempty bounded and perfect set $P\subset \mathbb R$, the bounded components of its complement are called \textsl{gaps} or $P$\textsl{-gaps}. Since the sets of subsums are bounded and perfect (\cite{Kak14}, \cite{Kak14a}, \cite{P}), understanding their gaps amounts to understanding their geometric properties. There are three very useful  facts about the gaps  and we will formulate them below for reader's convenience.

\begin{FGP}(\cite[Fact 21.9]{BFP})
If $a_k>r_k$ for some index $k$, then the open interval $(r_k,\,a_k)$ is a gap of $E$.
\end{FGP}

Sometimes, like in the following fact, it is convenient to arrange all elements of $F_k$ into an increasing finite sequence $(f_j^{(k)})_{j=0}^{m(k)}$.

\begin{SGP}(\cite[Lemma 2]{NS00}, \cite[Fact 21.11]{BFP})
 Let $(a,\,b)$ be an $E$-gap and let $k=\max\{n:\ a_n\ge b-a\,\}$. Then $b\in F_k$. Moreover, if $b=f_j^{(k)}$, then $a=f_{j-1}^{(k)}+r_k$.
 \end{SGP}

 \begin{TGP} (\cite[Lemma 2.4]{BGM}, \cite[Lemma 4]{BGFPS})
 Suppose that $(a,\,b)$ is an $E$-gap such that all $E$-gaps lying to the left of it are shorter than $b-a$. Then $b=a_n$ and $a=r_n$ for some $n\in\mathbb N$.
 \end{TGP}

\section{The general case}

Let $A$ be a nonempty subset of a metric space $(X,\,d)$. The center of distances $C(A)$ of the set $A$ is defined by
$$
C(A)\ :=\ \left\{\alpha\in[0,\,+\infty):\ \forall\ x\in A\ \ \exists\ y\in A\ \ d(x,\,y)=\alpha\ \right\}.
$$
Evidently, $0\in C(A)$ always.

Now, let $(X,\,d)$ be a separable metric space and let $A\subset X$ be a nonempty closed subset. A sequence $(a_n)$ of elements of $X$ will be said to be concentrated on $A$ if $A$ is the set of all accumulation points of the sequence.

\begin{fact}
\label{f11}
For every nonempty closed set $A$ in a separable metric space $X$ there is a sequence $(a_n)$ with values in $A$ and concentrated on A.
\end{fact}
\begin{proof}
Let $\{U_n:\ n\in\mathbb N\,\}$ be a countable basis of $(X,\,d)$. For every $n\in\mathbb N$ such that $U_n\cap A\ne\emptyset$ choose one element $a_n$ of the intersection. The resulting sequence  $(a_n)$ is concentrated on $A$.
\end{proof}
The well-known theorem of John von Neumann says that two sequences $(a_n)$ and $(b_n)$ are concentrated on the same set if and only if there is a permutation $\pi:\,\mathbb N\to\mathbb N$ such that $d(a_n,\,b_{\pi(n)})\to 0$ (\cite{vN}, \cite{H}).

\begin{thm}
\label{t12}
Let $A$ be a closed subset of a compact metric space $(X,\,d)$ and let $(a_n)$, $(b_n)$ be two sequences concentrated on $A$. Then
$$
C(A)\ =\ \left\{ \alpha\in[0,\,+\infty):\ \text{there is a permutation $\pi:\,\mathbb N\to\mathbb N$}\ \ d(a_n,\,b_{\pi(n)})\to\alpha\:\right\}.
$$
\end{thm}
\begin{proof}
 The nontrivial inclusion $\subseteq$ has been proved by Bielas, Plewik and Walczy\'{n}ska \cite[Thm 1.]{BPW}. The reversed inclusion $\supseteq$ is easy. Indeed, take any $x\in A$ and let $d(a_n,\,b_{\pi(n)})\to\alpha$. The sequence $(a_n)$ is concentrated on $A$ and so $a_{n_m}\to x$. Since $X$ is compact, the sequence $(b_{\pi(n_m)})_{m\in\mathbb N}$ as a subsequence $\bigl( b_{\pi(n_{m_k})}\bigr)_{k\in\mathbb N}$ convergent to some $y\in A$. Then $d(x,\,y)=\lim_{k\to\infty}d(a_{n_{m_k}},\,b_{\pi(n_{m_k})})=\alpha$.
\end{proof}

\begin{cor}
\label{c13}
Let $A$ be a closed subset of a compact metric space $(X,\,d)$ and let $(a_n)$ be any sequence concentrated on $A$. Then
$$
C(A)\ =\ \left\{ \alpha\in[0,\,+\infty):\ \text{there is a permutation $\pi:\,\mathbb N\to\mathbb N$}\ \ d(a_n,\,a_{\pi(n)})\to\alpha\:\right\}.
$$
\end{cor}

\begin{fact}
\label{f14}
If $A$ is nonempty and compact, then $C(A)$ is compact.
\end{fact}
\begin{proof}
Given any $x\in X$, the function $f_x:\,A\to[0,\,+\infty):\,y\mapsto d(x,\,y)$ is continuous and hence $f_x(A)$ is compact. The thesis follows from the equality $C(A)=\bigcap_{x\in A}f_x(A)$.
\end{proof}

\begin{fact}
\label{f15}
For any finite family of nonempty sets $\{A_i:\ i=1,\,\ldots,\,m\,\}$ the following inclusion holds
$$
\bigcap_{i=1}^mC(A_i)\ \subset\ C\left(\bigcup_{i=1}^mA_i\right).
$$
\end{fact}
\begin{proof}
Let $\alpha\in\bigcap_{i=1}^mC(A_i)$. Take any $x\in\bigcup_{i=1}^mA_i$ and let $n$ be such that $x\in A_n$. Since $\alpha\in C(A_n)$, we can find $y\in A_n\subset\bigcup_{i=1}^mA_i$ such that $d(x,\,y)=\alpha$. Since $x$ was arbitrary, $\alpha\in C\left(\bigcup_{i=1}^mA_i\right)$.
\end{proof}
The inverse inclusion fails in general. It suffices to consider the sets $A_1=\{0,\,1\}$ and $A_2=\{2,\,3\}$ in $\mathbb R$.

\begin{fact}
\label{f16}
For any descending sequence $(B_n)$ of nonempty compact sets the following inclusion holds
$$
\bigcap_{n=1}^\infty C(B_n)\ \subset \ C\left(\bigcap_{n=1}^\infty B_n\right).
$$
\end{fact}
\begin{proof}
Take any $\alpha\in\bigcap C(B_n)$ and any $x\in\bigcap B_n$. Then for every $n\in\mathbb N$ there is an $y_n\in B_n$ such that $d(x,\,y_n)=\alpha$. Since $B_1$ is compact, we may assume $y_n\to y$. It must be $y\in\bigcap B_n$ because $(y_k)_{k\ge n}\subset B_n$ and all $B_n$'s are compact. Clearly, $d(x,\,y)=\lim_nd(x,\,y_n)=\alpha$ and hence $\alpha \in C\left(\bigcap B_n\right)$.
\end{proof}

The inverse inclusion does not have to hold as the example of $B_n:=\{0,\,1\}\cup\{1+\frac1k:\ k\ge n\,\}$ in $\mathbb R$ shows. Moreover, the assumption of compactness of all $B_n$'s cannot be dropped in the general case as the example  of $B_n:=\{0\}\cup\{e_k:\ k\ge n\,\}$ in the space $c_0$ shows. We do not know if the assumption of compactness is necessary in $\mathbb R$ when $\bigcap B_n$ is required additionally.

\section{The case of subsums}

We start the section with two obvious observations regarding the center of distances for subsets of $\mathbb R$.

\begin{fact}
\label{f21}
If $A\subset[0,\,+\infty)$ and $0\in A$, then $C(A)\subset A$.
\end{fact}

\begin{fact}
\label{f22}
If $q\in\mathbb R$, then $C(q+A)=C(A)$ and $C(qA)=|q|C(A)$.
\end{fact}

The symbol $d(x,\,A)$ denotes the distance of the point $x$ from the set $A$, that is, the number $\inf_{y\in A}d(x,\,y)$.

\begin{fact}
\label{f23}
Given a convergent series $\sum a_n$ of positive terms and of sum $s$, the following inclusions are true
$$
\{0\}\cup\{a_n:\ n\in\mathbb N\,\}\ \subset \ C(E)\ \subset \ [0,\,\tfrac s2+d(\tfrac s2,\,E)]\cap E.
$$
\end{fact}
\begin{proof}
Both inclusions are rather easy to observe and we will prove the second one only. Setting $\alpha:=\max\{x\in E:\ x\le\frac s2\,\}$ and $\beta:=\min\{x\in E:\ x\ge\frac s2\,\}$, we get $\max_{x\in E}|\alpha-x|=s-\alpha$ and $\max_{x\in E}|\beta-x|=\beta$. Thus
$$
C(E)\ \subset\ [0,\,s-\alpha]\cap[0,\,\beta]\ =\ \bigl[0,\,\tfrac s2+\min\{\tfrac s2-\alpha,\,\beta-\tfrac s2\}\bigr]\ =\ [0,\,\tfrac s2+d(\tfrac s2,\,E)].
$$
The inclusion $C(E)\subset E$ follows from the Fact \ref{f21}.
\end{proof}

Our next observation is very easy as well and so its proof, based on the First Gap Lemma and on one of basic Kakeya's results \cite[Facts 21.9 and 21.13]{BFP}, will be leaved out.
\begin{fact}
\label{f24}
Given a convergent series $\sum a_n$ of positive terms and of sum $s$, the following statements are equivalent:
\begin{itemize}
\item[(i)] $C(E)$ is an inverval,
\item[(ii)] $C(E)\ =\ [0,\,\frac s2]$.
\item[(iii)] the series $\sum a_n$ is slowly convergent.
\end{itemize}
\end{fact}

\begin{fact}
\label{f25}
Given a convergent series $\sum a_n$ of positive terms and a positive integer $k$, the following inclusions hold
$$
C(F_k)\ \ \subset\ \ C(F_{k+1})\ \ \subset\ \ C(E).
$$
\end{fact}
\begin{proof}
Take any $z\in C(F_k)$. For every $f\in F_k$ there is an $\tilde{f}\in F_k$ such that $|f-\tilde{f}|=z$. Now, take any $g\in F_{k+1}$ and any set $A\subset\{1,\,\ldots,\,k+1\,\}$ such that $g=\sum_{n\in A}a_n$. Define
$$
f\ :=\ \sum_{\substack{n\in A\\ n\le k}}a_n\hspace{.3in}\text{and}\hspace{.3in} \hat{g}\ :=\ \begin{cases} \tilde{f}\ \ \ \ &\text{if $g=f$}\\ \tilde{f}+a_{k+1} &\text{otherwise.} \end{cases}
$$
Clearly, $\hat{g}\in F_{k+1}$ and $|g-\hat{g}|=z$ which proves the first inclusion.

In order to prove the second one, take any $z\in C(F_k)$ again. Given an $x\in E$, there are $f_x\in F_k$ and $e_x\in E_k$ such that $x=f_x+e_x$ \cite[Fact 21.5]{BFP}. Since $z\in C(F_k)$, we can find $f\in F_k$ such that $|f-f_x|=z$. Setting $y:=f+e_x$, we get $y\in E$ and $|x-y|=z$ which completes the proof of the second inclusion.
\end{proof}

Is it possible that $C(F)=\bigcup_nC(F_n)$ ? Yes, it is sometimes possible as we will show at the very end of this note. However, the equality does not have to hold always which can be seen from our next elementary observation.
\begin{fact}
\label{f26}
If $E(a_n)$ is a Cantorval containing a middle interval $I$ of length $|I|>\frac s2$, then $\frac12[s-|I|,\,|I|]\subset C(E)$.
\end{fact}
Such Cantorvals do exist. For example, it suffices to take the Ferens type series $\mathcal{F}(2,5;\frac1{16})$ (see \cite[Lemma 7]{BP}). Since  the set $\bigcup_nC(F_n)$ is countable for any series $\sum a_n$, the equality $C(E)=\bigcup_nC(F_n)$ does not hold for any series satisfying the assumptions of the Fact \ref{f26}.
\begin{op}
\label{op1}
For which series $\sum a_n$ does the equality $C(E)=\overline{\bigcup_{n=1}^\infty C(F_n)}$ hold?
\end{op}

\begin{fact}
\label{f27}
Given a positive integer $k$ and a convergent series  $\sum a_n$ of positive terms, the following inclusions hold
$$
C(E_{k+1})\ \subset\ C(E_k)\ \subset\ C(E).
$$
\end{fact}
\begin{proof}
We are going to prove the left inclusion only, because the right one can be proven analogously:
$$
C(E_{k+1}) = C(E_{k+1})\cap C(E_{k+1}) \overset{\text{Fact \ref{f22}}}{=} C(E_{k+1})\cap C(a_{k+1}+E_{k+1})\overset{\text{Fact \ref{f15}}}{\subset} C(E_k).
$$
\end{proof}

\section{The case of a fast convergent series}

It is well known that given a fast convergent series $\sum a_n$ and an $x\in E$, there is a unique set of indices $I_x\subset\mathbb N$ such that $x=\sum_{n\in I_x}a_n$ \cite[Fact 3.3]{P}. In particular,
\begin{equation}
\label{e31}
x\in E\setminus F \hspace{.2in}\text{if and only if} \hspace{.2in} \overline{\overline{I_x}}=\infty.
\end{equation}

\begin{fact}
\label{f31}
If a series $\sum a_n$ is fast convergent, then $C(E)\subset F\cap[0,\,a_1]$.
\end{fact}
\begin{proof}
It suffices to show that $(E\setminus F)\cap[0,\,a_1]\cap C(E)=\emptyset$, because of \eqref{e31} and the Fact \ref{f23}. Take any $\alpha\in(E\setminus F)\cap[0,\,a_1]$. We search for an $x\in E$ such that $x+\alpha \not\in E$ and $x-\alpha\not\in E$. Let $I_\alpha$ be the unique set of indices such that $\alpha=\sum_{n\in I_\alpha}a_n$. Define $m:=\min I_\alpha$. Since $\alpha<a_1$, we get $m\ge2$. Furhter, define $k:=\min\{n\in I_\alpha:\ a_n<a_{m-1}-r_{m-1}\,\}$, $I:=\{k\}\cup\mathbb N\setminus\bigl(\{1,\,2,\,\ldots,\,m-1\}\cup I_\alpha\bigr)$ and $x:=\sum_{n\in I}a_n$. Clearly, $x-\alpha<r_m-a_{m-1}<0$ and hence $x-\alpha\not\in E$. Finally, we get $r_{m-1}<r_{m-1}+a_k=x+\alpha<a_{m-1}$ where the last inequality is a consequence of our choice of $k$. Thus, $x+\alpha\not\in E$.
\end{proof}

The next observation is very similar to the Third Gap Lemma (see \cite[Lemma 2.4]{BGM} or \cite[Lemma 4]{BGFPS}) and therefore we like to call it the Fourth Gap Lemma.

\begin{lem}
\label{l32}
Let $\sum a_n$ be any fast convergent series and let $l>0$ be the length of any gap of $E(a_n)$. Then the most left of all $E$-gaps of length $l$ is exactly the gap $(r_k,\,a_k)$ where $k=\max\{ n:\ a_n-r_n=l\,\}$.
\end{lem}
\begin{proof}
It is known that if $(a,\,b)$ is an internal gap of $E(a_n)$, then $b\in F$ and $a=b-a_{m_b}+r_{m_b}$ where $m_b:=\max I_b$ (cf. \cite[Corollary II.3.5]{P} or \cite[Lemma 3]{BFP2}). Suppose that $(\alpha,\,\beta)$ is an $E$-gap of length $l$ lying to the left of the gap $(r_k,\,a_k)$. Then $\beta<r_k$ and hence $I_\beta\subset\{k+1,\,k+2,\,\ldots\}$. Thus $m_\beta>k$. However,
$$
l\ =\ \beta-\alpha\ =\ \beta-(\beta-a_{m_\beta}+r_{m_\beta})\ =\ a_{m_\beta}-r_{m_\beta}
$$
which contradicts the definition of $k$.
\end{proof}

\begin{exa}
\label{exa33}
If we managed to drop the assumption on fast convergence from the Lemma \ref{l32}, we would obtain a statement that implies both the Third and the Fourth Gap Lemmas. However, it is impossible. To see that consider any series $\sum a_n$ satisfying the following four requirements: $a_1=\frac35$, $a_2=\frac12$, $a_3=\frac25$ and $r_3=\frac1{10}$. Then $(\frac7{10},\,\frac9{10})$ is the only  $E$-gap of  length $\frac2{10}$ and it is not of the form $(r_k,\,a_k)$ for any $k$.
\end{exa}

\begin{thm}
\label{t34}
Let $\sum a_n$ be a fast convergent series. If
\begin{equation}
\label{34star}
\max_{k\ge n+2}(a_k-r_k) \ <\ a_n-r_n \hspace{.3in}\text{for all $n$},
\end{equation}
then
\begin{equation}
\label{34sstar}
C(E)\ =\ \{0\}\cup\{a_n:\ n\in\mathbb N\,\}.
\end{equation}
\end{thm}
\begin{proof}
Because of the Fact \ref{f23}, it remains to show that $C(E)\subset\{0\}\cup\{a_n:\ n\in\mathbb N\,\}$. Suppose it is not true. Then there are a $y\in C(E)$ and a positive integer $n$ such that $y\in (a_{n+1},\, r_n]$. Clearly, $y\in E$ by the Fact \ref{f21}. For $x\in E\cap[0,\,y)$ points of the form $x-y$ do not belong to $E$ which implies
\begin{equation}
\label{341}
y\ +\ \bigl(E\cap[0,\,y)\bigr)\ \subset \ E,
\end{equation}
because $y\in C(E)$. We have
\begin{equation}
\label{342}
r_n\ \ge\ y\in\bigl(E\cap[0,\,y)\bigr)+y.
\end{equation}
Also $a_{n+1}+y\in\bigl(E\cap[0,\,y)\bigr)+y$ and $a_{n+1}+y>a_{n+1}+r_{n+1}=r_n$. Hence, in the view of \eqref{341}
\begin{equation}
\label{343}
a_{n+1}+y\ \ge\ a_n.
\end{equation}
By \eqref{342} and \eqref{343}, there must be a gap in the set $\bigl(E\cap[0,\,y)\bigr)+y$ containing the gap $(r_n,\,a_n)$. Now, by the Fourth Gap Lemma (Lemma \ref{l32}) there is $k$ such that the gap $(r_k,\,a_k)$ of $E\cap[0,\,y)$ is of length at least $a_n-r_n$. It cannot be the gap $(r_{n+1},\, a_{n+1})$, because $r_{n+1}+y>a_{n+1}+r_{n+1}=r_n$. There for, $k\ge n+2$ which contradicts one of the main assumptions of the Theorem \ref{t34}.
\end{proof}

\begin{cor}{\cite[Thm. 4]{BPW}}
\label{c35}
If $q\in(0,\,\frac12)$ and $a>0$, then the center  of distances of the set of subsums of the geometric series $\sum_{n=1}^\infty aq^n=\sum(a;\,q)$ consists only from 0 and the terms of the series.
\end{cor}
\begin{proof}
Denoting $a_n:= aq^n$, we get
$$
a_n-r_n\ =\ aq^n\frac{1-2q}{1-q}.
$$
The function $f:[1,\,+\infty)\to\mathbb R:\, x\mapsto aq^x\frac{1-2q}{1-q}$ is decreasing, and thus the assumptions of the Theorem \ref{t34} are fulfilled.
\end{proof}

\begin{exa}
\label{exa36}
The Thm. \ref{t34} resulted from our failed attempt to prove that the equality \eqref{34sstar} holds for every fast convergent series. We know now that the hypothesis is false. Indeed, consider the fast convergent multigeometric series $\sum(4,\,2,\,1;\frac1{100})$. Given any $k\in\mathbb N$, we get
$$
a_{3k-2}-r_{3k-2}\ =\ a_{3k-1}-r_{3k-1}\ =\ a_{3k}-r_{3k}\ =\ \frac1{100^{k-1}}-\frac7{100^{k-1}\cdot99}.
$$
We are going to show that $3\in C(E)$. Take any $x\in E$. $x=\sum_{n\in I}a_n$ for a unique set of indices $I\subset\mathbb N$. We want to find $y=\sum_{n\in J}a_n\in E$ such that $|x-y|=3$.

If $1\in I$ and $3\not\in J$, take $J:=I\cup\{3\}\setminus\{1\}$.

If $\{1,\,2,\,3\}\subset I$, take $J:=I\setminus\{2,\,3\}$.

If $\{1,\,3\}\subset I$ and $2\not\in I$, take $J:=I\cup\{2\}\setminus\{1,\,3\}$.

If $1\not\in I$ and $3\in I$, take $J:=I\cup\{1\}\setminus\{3\}$.

If $1\not\in I$, $3\not\in I$ and $2\in I$, take $J:=I\cup\{1,\,3\}\setminus\{2\}$.

If $\{1,\,2,\,3\}\cap I=\emptyset$, take $J:=I\cup\{2,\,3\}$.

Note, in a similar manner, that every number of the form $\frac3{100^k}$, $k\in\mathbb N$, belongs to $C(E)$, not being a term of the series. A question raises if $C(E)=\{0\}\cup\{a_n: n\in\mathbb N\}\cup\{\frac3{100^k}:\ k\in\mathbb N\,\}$. We will be able to answer it positively later in the paper.
\end{exa}

\begin{exa}
The series from the previous example does not fulfill the condition \eqref{34star} for any $n$ divisible by 3. In fact, the condition \eqref{34star} is not necessary for the equality \eqref{34sstar}. To see this, consider the fast convergent multigeometric series $\sum(11,\,6,\,3:\,q)$ with $q$ such that $\sum_{n=1}^\infty q^n < \frac1{20}$. For every $k\in\mathbb N$, we get
$$
r_{3k-2}=9q^{k-1}+20\sum_{i=k}^\infty q^i, \hspace{.2in} r_{3k-1}=3q^{k-1}+20\sum_{i=k}^\infty q^i, \hspace{.2in} r_{3k}=20\sum_{i=k}^\infty q^i.
$$
Further,
$$
a_{3k-2}-r_{3k-2}\ =\ q^{k-1}\bigl(2-20\sum_{i=1}^\infty q^i\bigr),$$
$$
 a_{3k}-r_{3k}\ =\ a_{3k-1}-r_{3k-1}\ = \ q^{k-1}\bigl(3-20\sum_{i=1}^\infty q^i\bigr),
$$
for any $k\in\mathbb N$, and hence $a_n-r_n<\max_{k\ge n+2}(a_k-r_k)$ for any $n\in3\mathbb N-2$ which means that \eqref{34star} does not hold for the series.

It remains to show that the center of distances of the series does not contain anything but 0 and all terms of the series. Take any $\alpha\in\bigl(F\cap[0,\,11]\setminus\bigl(\{0\}\cup\{a_n: n\in\mathbb N\,\}\bigr)$. There exists $I_\alpha\subset\mathbb N\setminus\{1\}$ such that $2\le\overline{\overline{I_\alpha}}<+\infty$ and $\alpha=\sum_{n\in I_\alpha}a_n$.
 Define $m:=\min I_\alpha \ge2$, $M:=\max I_\alpha$ and $J:=\{1,\,2,\,\ldots,\,m-1\,\}\cup I_\alpha$. Then define $x:=\sum_{n\in(\mathbb N\setminus J)\cup\{M\}}a_n$. The inequalites $x-\alpha<r_m-a_m<0$ yield $x-\alpha\not\in E$ easily.

In order to show that $x+\alpha\not\in E$, observe first that
\begin{equation}
\label{e371}
x+\alpha\ =\ r_{m-1}+a_M.
\end{equation}
We have to analyse three cases depending on the form of the remainder $r_{m-1}$:
\begin{itemize}
\item[(a)] $\displaystyle r_{m-1}\ =\ 9q^{k-1}+20\sum_{n=k}^\infty q^n\ \text{for some $k\in\mathbb N$ (when $m=3k-1$),}$
\item[(b)] $\displaystyle r_{m-1}\ =\ 3q^{k-1}+20\sum_{n=k}^\infty q^n\ \text{for some $k\in\mathbb N$ (when $m=3k$),}$
\item[(c)] $\displaystyle r_{m-1}\ =\ 20\sum_{n=k}^\infty q^n\ \text{for some $k\in\mathbb N$ (when $m=3k-2$).}$
\end{itemize}

In the case (a) $a_M=3q^{k-1}$ or $a_M\in\bigl\{11q^{k-1+i},\,6q^{k-1+i},\,3q^{k-1+i}:\ i\in\mathbb N\,\bigr\}$.  If $a_M=3q^{k-1}$, then the inequalities
\begin{align*}
a_{3k-2}+r_{3k}\ &=\ 11q^{k-1}+20\sum_{n=k}^\infty q^n\ <\ x+\alpha\ \overset{\eqref{e371}}{=} 12q^{k-1}+20\sum_{n=k}^\infty q^n\\
&<\ 11q^{k-1}+3q^{k-1}\ =\ a_{3k-2}+a_{3k}
\end{align*}
imply that $x+\alpha\in(a_{3k-2}+r_{3k},\,a_{3k-2}+a_{3k})$, that is, $x+\alpha\not\in E$. If $a_M$ is of the form $11q^{k-1+i}$ or $6q^{k-1+i}$, or $3q^{k-1+i}$ for some $i\in\mathbb N$, then
$$
r_{3k-2}=9q^{k-1}+20\sum_{n=k}^\infty q^n\ \overset{\eqref{e371}}{<}\ x+\alpha\ \overset{\eqref{e371}}{<}\ r_{m-1}+q^{k-1}<11q^{k-1}=a_{3k-2}
$$
and hence $x+\alpha\in(r_{3k-2},\,a_{3k-2})$ which implies that $x+\alpha\not\in E$.

In the case (b) $a_M=11q^{k-1+i}$ or $6q^{k-1+i}$, or $3q^{k-1+i}$ for some $i\in\mathbb N$. Thus
$$
r_{3k-1}\ <\ 3q^{k-1}+20\sum_{n=k}^\infty q^n+a_M\ =\ x+\alpha\ <\ 3q^{k-1}+q^{k-1}+q^{k-1}<6q^{k-1}=a_{3k-1}
$$
and hence $x+\alpha\not\in E$.

In the case (c) we get $a_M=6q^{k-1+i}$ or $3q^{k-1+i}$ for some $i\in\mathbb N$ or $a_M=11q^{k-1+i}$ for some $i\ge2$. Thus
$$
r_{3k}\ <\ 20\sum_{n=k}^\infty q^n+a_M\ <\ q^{k-1}+q^{k-1}\ <\ 3q^{k-1}\ =\ a_{3k}
$$
and hence $x+\alpha\not\in E$.

We have proved that $\alpha\not\in C(E)$ and it must be $ C(E)=\{0\}\cup\{a_n:\ n\in\mathbb N\,\}$ by the Fact \ref{f23}.
\end{exa}

The problem of characterizing those fast convergent series for which \eqref{34sstar} holds remains open. A complete characterization of \eqref{34sstar} will be even more difficult, because there are series that satisfy \eqref{34sstar} and are not fast convergent (\cite[Thm. 11]{BPW}).

\section{The case of a multigeometric series}

A multigeometric series $\sum(l_1,\,\ldots,\,l_m;q)$ converges to the sum $s=\frac1{1-q}\sum_{i=1}^m l_i$. It is not difficult to see that for $E=E(l_1,\,\ldots,\,l_m;q)$ and for any $k\in\mathbb N$
\begin{equation}
\label{401}
q^kE\ =\ E_{km}.
\end{equation}

\begin{thm}
\label{t41}
Let a multigeometric series $\sum(l_1,\,\ldots,\,l_m;q)$ satisfies the properties
\begin{equation}
\label{p1}
2qs\ <\ \delta(F_m)
\end{equation}
and
\begin{equation}
\label{p2}
qs\ <\ 1
\end{equation}
where $\delta(F_m):=\min\bigl\{|f-g|:\ f,\,g\in F_m,\ f\not=g\,\bigr\}$. Then
\begin{equation}
\label{e411}
C(E)\ =\ \bigcup_{k=0}^\infty q^kC(F_m).
\end{equation}
\end{thm}

Although the single inequality $2qs<1$ implies both \eqref{p1} and \eqref{p2}, it is not equivalent to the conjunction of these properties.
\begin{proof}
First, we want to observe that
\begin{equation}
\label{41a}
C(E)\cap[l_m,\,s]\ \subset F_m\setminus\{0\}.
\end{equation}
Indeed, if $x\in C(E)\cap[l_m,\,s]$, then $qs-x<0$ by the property \eqref{p2}. Hence, in the light of $qs\in E$, it must be $qs+x\in E$. Thus by the Fact \ref{f21}
\begin{equation}
\label{412}
C(E)\cap[l_m,\,s]\subset\bigl\{x\in[l_m,\,s]:\ qs+x\in E\,\bigr\}\cap E\subset \bigl(E\cap[l_m,\,s]\bigr)-qs.
\end{equation}
The Fact 21.8 of \cite{BFP} together with the assumption \eqref{p1} yield
$$
E\ \subset\ \bigsqcup_{f\in F_m}[f,\,f+qs].
$$
Therefore,
$$
A:=\bigsqcup_{f\in F_m\setminus\{0\}}[f-qs,\,f]=\left( \bigsqcup_{f\in F_m\setminus\{0\}}[f,\,f+qs]\right)-qs\overset{\eqref{412}}{\supset}C(E)\cap[l_m,\,s]
$$
and
$$
B:=\bigsqcup_{f\in F_m\setminus\{0\}}[f,\,f+qs]\overset{\eqref{412}}{\supset}E\cap[l_m,s]\overset{\text{Fact \ref{f21}}}{\supset}C(E)\cap[l_m,\,s].
$$
The assumption \eqref{p1} implies that $A\cap B=F_m\setminus \{0\}$ which proves \eqref{41a}.

Next, observe that
\begin{equation}
\label{41b}
C(E)\cap[l_m,\,s]\subset C(F_m)\setminus\{0\}.
\end{equation}
Indeed, suppose that the inclusion does not hold. There exists $z\in C(E)\cap[l_m,\,s]$ such that $z\not\in C(F_m)$. Hence there is an $f\in F_m$ such that both $f+z$ and $f-z$ do not belong to $F_m$. Since $f+z\not\in F_m$ and since $z\in F_m\setminus\{0\}$ by \eqref{41a}, we get $f+z\in\mathbb N\setminus F_m$, because $F_m\subset\mathbb N$. Therefore, by \eqref{p2},
$$
d(f+z,\,E)\ \ge \ d\left(f+z,\,\bigcup_{\phi\in F_m}(\phi+[0,\,qs])\right)\ >\ 1-qs>0.
$$
Hence $f+z\not\in E$.

Analogously, since $f-z\not\in F_m$, we get $f-z\in\mathbb Z\setminus F_m$, and thus, $d(f-z,\,E)>0$, that is, $f-z\not\in E$ which together with $f+z\not\in E$ yields $z\not\in C(E)$, a contradiction.

Our third observation
\begin{equation}
\label{41c}
C(E)\cap(qs,\,s]\ \subset\ C(F_m)\setminus\{0\}
\end{equation}
follows from \eqref{41b}, the Fact \ref{f21}, because $(qs,\,l_m)$ is an $E$-gap by the assumption \eqref{p2} and by the First Gap Lemma \cite[Fact 21.9]{BFP}.

Next, observe that
\begin{equation}
\label{41d}
\forall k\in\mathbb N_0 \hspace{.3in} C(E)\cap(q^{k+1}s,\,q^ks]\ \subset\ q^kC(F_m).
\end{equation}
Indeed, take any $k\ge 1$ and $ z\in C(E)\cap(q^{k+1}s,\,q^ks]$. Since
$$
\min\,\bigcup_{f\in F_{km}\setminus\{0\}}(f+E_{km})\ = \ q^{k-1}l_m\ \overset{\eqref{p1}}{>}\ q^ks,
$$
we get $z\in E_{km}$. Now, because $z\in C(E)$ and
$$
E_{km}+E_{km}\ \subset[0,\,2q^ks]\ \overset{\eqref{p1}}{\subset}\ [0,\,q^{k-1}l_m),
$$
it must be $z\in C(E_{km})$. Hence
$$
\frac1{q^k}z\in\frac1{q^k}C(E_{km})\ \overset{\text{Fact \ref{f22}}}{=}\ C\bigl(\frac1{q^k}E_{km})\ \overset{\eqref{401}}{=}\ C(E).
$$
Moreover, $\frac1{q^k}z\in(qs,\,s]$ and thus $\frac1{q^k}z\in C(F_m)\setminus\{0\}$ by \eqref{41c}. Finally,
\begin{align*}
C(E)\ &=\ \{0\}\cup\bigcup_{k=0}^\infty(q^{k+1}s,\,q^ks]\cap C(E)\ \overset{\eqref{41d}}{\subset}\ \bigcup_{k=0}^\infty q^kC(F_m)\\
&\overset{\text{Fact \ref{f25}}}{\subset}\ \bigcup_{k=0}^\infty q^kC(E)\ =\ \bigcup_{k=0}^\infty C(q^kE)\\
&\overset{\eqref{401}}{=}\ \bigcup_{k=0}^\infty C(E_{km})\ \overset{\text{Fact \ref{f27}}}{\subset}\ \bigcup_{k=0}^\infty C(E)\ =\ C(E).
\end{align*}
\end{proof}

\begin{exa}
\label{exa42}
The Thm. \ref{t41} provides the positive answer to the question raised at the end of our Example \ref{exa36}. Indeed, for the series $\sum(4,\,2,\,1;\frac1{100})$ we obtain $s=\frac{700}{99}$, $2qs=\frac{14}{99}<1$, $m=3$. Hence both properties \eqref{p1} and \eqref{p2} are satisfied. Since $F_3=\{0,\,1,\,2,\,\ldots,\,7\,\}$, we get $\delta(F_m)=1$ and $C(F_m)=\{0,\,1,\,2,\,3,\,4\,\}$. Thus, by the Thm. \ref{t41},
$$
C(E)\ =\ \{0\}\cup\bigl\{\frac m{100^k}:\ m\in\{1,\,2,\,3,\,4\,\},\ k\in\mathbb N_0\,\bigr\}\  \varsupsetneq\ \{0\}\cup\{a_n:\ n\in\mathbb N\,\}.
$$
\end{exa}

\begin{exa}
\label{exa43}
The Thm. \ref{t41} can be applied to the series $\sum(11,\,6,\,3;q)$ for $q<\frac1{21}$ as well. This time $2qs$ is larger than 1 for $q>\frac1{41}$, but nevertheless \eqref{p1} and \eqref{p2} hold for all $q<\frac1{21}$. Clearly, $s<21$, $m=3$. Since $F_3=\{0,\,3,\,6,\,9,\,11,\,14,\,17,\,20\,\}$, we get $\delta(F_3)=2$ and hence $2sq<\delta(F_m)$. We see easily that $C(F_3)=\{0,\,3,\,6,\,11\,\}$ an thus, by the Thm. \ref{t41},
$$
C(E)\ =\ \bigl\{mq^k:\ m\in\{0,\,3,\,6,\,11\,\},\ k\in\mathbb N_0\,\bigr\}\ =\ \{0\}\cup\{a_n:\ n\in\mathbb N\,\}.
$$
\end{exa}

\begin{lem}
\label{l44}
For every multigeometric series $\sum(l_1,\,\ldots,\,l_m;q)$ and every $i\in \mathbb N$ the equality $q^{i-1}C(F_m)\subset C(F_{im})$ holds.
\end{lem}
\begin{proof}
Let $z\in C(F_m)$. Take any $f\in F_{im}$. Let $f'\in F_{(i-1)m}$ and $f''\in q^{i-1}F_m$ be such that $f=f'+f''$. Since $z\in C(F_m)$, we can find $f'''\in F_m$ such that $|q^{1-i}f''-f'''|=z$. Then, for $\tilde{f}:=f'+q^{i-1}f'''$ we get $\tilde{f}\in F_{im}$ and $|f-\tilde{f}|=q^{i-1}z$.
\end{proof}

\begin{cor}
\label{c45}
Under the assumptions of the Thm. \ref{t41} the equality
$$
C(F_{km})\ =\ \bigcup_{i=1}^k\,q^{i-1}C(F_m)
$$
holds for every $k\in\mathbb N$.
\end{cor}
\begin{proof}
The Fact \ref{f21} implies that $C(F_{km})\subset F_{km}$. Hence
$$
C(F_{km})\ =\ \{0\}\cup\left(C(F_{km})\cap(q^ks,\,s]\right),
$$
and thus
\begin{align*}
C(F_{km})\ &\overset{\text{Fact \ref{f25}}}{\subset}\ \{0\}\cup\bigcup_{i=1}^k\bigl(C(E)\cap(q^is,\,q^{i-1}s]\bigr)\ \overset{\eqref{41d}}{\subset}\ \bigcup_{i=1}^kq^{i-1}C(F_m)\\
&\overset{\text{Lemma \ref{l44}}}{\subset}\ \bigcup_{i=1}^k C(F_{im})\ \overset{\text{Fact \ref{f25}}}{\subset}\ C(F_{km}).
\end{align*}
\end{proof}

\begin{cor}
\label{c46}
Under the assumptions of the Thm. \ref{t41} the equality
\begin{equation}
\label{461}
C(E)\ =\ \bigcup_{n=1}^\infty C(F_n)
\end{equation}
holds.
\end{cor}
\begin{proof}
$$
\bigcup_{n=1}^\infty C(F_n)\ \overset{\text{Fact \ref{f25}}}{=}\ \bigcup_{k=1}^\infty C(F_{km})\ \overset{\text{Cor. \ref{c45}}}{=}\ \bigcup_{i=1}^\infty q^{i-1}C(F_m) \ \overset{\text{Thm. \ref{t41}}}{=}\ C(E).
$$
\end{proof}

Let us conclude the note with the observation that there are series not fulfilling the requirements of the Thm. \ref{t41} and yet satisfying the equality \eqref{461}. The Guthrie-Nymann series $\sum(3,\,2;\frac14)$ has the sum $s=6\frac23$ and $ \delta(F_2)=1$. Hence it does not satisfy neither the condition \eqref{p1} nor \eqref{p2}. Nevertheless. the equality \eqref{461} is true for the series, since $C(F_2)=\{0,\,2,\,3\,\}$ and $C(E)=\{0\}\cup\bigl\{\frac k{4^n}:\ k\in\{2,\,3\,\}, n\in\mathbb N\,\bigr\}$ \cite[Thm. 11]{BPW}.

\end{document}